\documentclass[12pt]{amsart} 
\usepackage{amsmath,amsfonts,amssymb,euler}

\makeatletter
\@namedef{subjclassname@2010}{%
  \textup{2010} Mathematics Subject Classification}
\makeatother

\begin{document}

\theoremstyle{plain} 
\theoremstyle{definition} 
\theoremstyle{remark}

\newtheorem{prop}{Proposition} 
\newtheorem{thm}{Theorem}[section] 
\newtheorem{lem}[thm]{Lemma} 
\newtheorem{state}{State}[section]

\theoremstyle{definition} 
\newtheorem{definition}{Definition}[section]

\theoremstyle{remark} 
\newtheorem{remark}{Remark}[section] 
\newtheorem{example}{Example}

\title{Generalized LCM matrices}
\author{Antal Bege}
\address{Department of Mathematics and Informatics,
University of Sapientia,
P.O. 9, P. O. Box 4,
RO-540485 T\^\i rgu Mure\c{s},
Romania}
\email{abege@ms.sapientia.ro}
\date{
August 24, 2011}


\begin{abstract} 
Let $f$ be an arithmetical function. The matrix $[f[i,j]]_{n\times n}$ given by the value of $f$ in least  common multiple of $[i,j]$, $f\big([i,j]\big)$ as its $i,\; j$ entry is called the least common multiple (LCM) matrix. We consider the generalization of this matrix where the elements are in the form $f\big(n,[i,j]\big)$ and $f\big(n,i,j,[i,j]\big)$.
\end{abstract}

\maketitle

\section{Introduction}
The classical Smith determinant was introduced in 1875 by H. J. S. Smith \cite{smith1} who also proved that
\begin{equation}
\label{1}
\det [(i,j)]_{n\times n}=
\left|
\begin{array}{cccc}
(1,1)&(1,2)&\cdots&(1,n)\\
(2,1)&(2,2)&\cdots&(2,n)\\
\cdots&\cdots&\cdots&\cdots\\
(n,1)&(n,2)&\cdots&(n,n)\\
\end{array}
\right|
=\varphi (1) \varphi(2)\cdots \varphi(n),
\end{equation}
where $(i,j)$ represents the greatest common divisor of $i$ and $j$, and $\varphi(n)$ denotes the  Euler totient function.
\\
The GCD matrix with respect to $f$ is
\[
[f(i,j)]_{n\times n}=
\left[
\begin{array}{cccc}
f((1,1))&f((1,2))&\cdots&f((1,n))\\
f((2,1))&f((2,2))&\cdots&f((2,n))\\
\cdots&\cdots&\cdots&\cdots\\
f((n,1))&f((n,2))&\cdots&f((n,n))\\
\end{array}
\right].
\]
There are quite a few generalized forms of GCD matrices, which can be found in several references 
\cite{bege1, bourque1,  hauk1, hong1, ovall1}.

H. J. S. Smith \cite{smith1} also evaluated the determinant of
\[
\big[[i,j]\big]_{n\times n}=
\left[
\begin{array}{cccc}
\mbox{\upshape[}1,1]&[1,2]&\cdots &[1,n] \\
\mbox{\upshape[}2,1]&[2,2]&\cdots &[2,n] \\
\cdots &\cdots &\cdots &\cdots \\
\mbox{\upshape[}n,1]&[n,2]&\cdots&[n,n] \\
\end{array}
\right],
\]
and proved that
\begin{eqnarray*}
\det \big[[i,j]\big]_{n\times n}&=&(n!)^2g(1)g(2)\cdots g(n)=\\
&=&\prod_{k=1}^N\varphi(k)\prod_{p\mid k}(-p).
\end{eqnarray*}
where 
$
g(n)=\frac{1}{n}\sum_{d\mid n}d\mu (d),
$
$\mu(n)$ being the classical M\"obius function.
\\
The structure of an LCM matrix $\big[[i,j]\big]_{n\times n}$ is the following (I. Korkee, P. Haukkanen \cite{korkee1})
\[
\big[[i,j]\big]_{n\times n}=
\mbox{ diag}\big(1,\; 2,\ldots ,n\big) A A^{T} 
\mbox{ diag}\big(1\; 2,\ldots ,n\big),
\]
where
$A=[a_{ij}]_{n\times n}$, 
\[
a_{ij}=
\left\{
\begin{array}{ccc}
\sqrt{g(j)},&\mbox{ if }&j\mid i\\
0,&\mbox{ if }&j \not| \;  i
\end{array}
\right. .
\]

The LCM matrix with respect to $f$ is
\[
[f[i,j]]_{n\times n}=
\left[
\begin{array}{cccc}
f([1,1])&f([1,2])&\cdots&f([1,n])\\
f([2,1])&f([2,2])&\cdots&f([2,n])\\
\cdots&\cdots&\cdots&\cdots\\
f([n,1])&f([n,2])&\cdots&f([n,n])\\
\end{array}
\right].
\]
Results concerning LCM matrices appear in papers S. Beslin \cite{beslin1},
K. Bourque, S. Ligh \cite{bourque2},  W. Feng, S. Hong, J. Zhao  \cite{feng1}
P. Haukkanen, J. Wang and J. Sillanp\"a\"a
\cite{hauk1}. 

In this paper we study matrices which have as variables the least common multiple and the indices
\[
\big[f(n,[i,j])\big]_{n\times n}=
\left[
\begin{array}{cccc}
f(n,[1,1])&f(n,[1,2])&\cdots&f(n,[1,n])\\
f(n,[2,1])&f(n,[2,2])&\cdots&f(n,[2,n])\\
\cdots&\cdots&\cdots&\cdots\\
f(n,[n,1])&f(n,[n,2])&\cdots&f(n,[n,n])\\
\end{array}
\right]
\]
and the more general form matrices
\[
\big[f(n,i,j,[i,j])\big]_{n\times n}\!\!\!=\!\!\!
\left[\!\!\!
\begin{array}{cccc}
f(n,1,1,[1,1])&f(n,1,2,[1,2])&\!\cdots \!&f(n,1,n,[1,n])\\
f(n,2,1,[2,1])&f(n,2,2,[2,2])&\cdots&f(n,2,n,[2,n])\\
\cdots&\cdots&\cdots&\cdots\\
f(n,n,1,[n,1])&f(n,n,2,[n,2])&\cdots&f(n,n,n,[n,n])\\
\end{array}\!\!\!
\right]
\]
\section{Generalized LCM matrices}

\begin{thm}
\label{3_1_2}

For a given totally multiplicative  aritmetical function $g(n)$ let
\[
f(n,[i,j])=g([i,j])\sum_{k\leq \frac{n}{[i,j]}}g(k).
\]
Then
\begin{equation}
\label{3_1_1}
\big[f(n,[i,j])\big]_{n\times n}=
C_n^{T}\mbox{ diag}\big(g(1),\; g(2),\ldots ,g(n)\big)C_n,
\end{equation}
where
$C_n=[c_{ij}]_{n\times n}$ 
\[
c_{ij}=
\left\{
\begin{array}{ccc}
1,&\mbox{ if }&j\mid i\\
0,&\mbox{ if }&j \not| \;  i
\end{array}
\right. .
\]
For a determinant we have
\begin{eqnarray}
\label{111}
\det \big[f(n,[i,j])\big]_{n\times n}=g(1)g(2)\cdots g(n).
\end{eqnarray}
\end{thm}

\begin{proof}
After multiplication, the general element of
$A=(a_{ij})_{n\times n}$,
\[
A=C_n^{T}\mbox{ diag}\big(g(1),\; g(2),\ldots ,g(n)\big)C_n
\]
is
\begin{eqnarray*}
a_{ij}=\sum_{k=1}^nc_{ki}g(k)c_{kj}=\sum_{\footnotesize
\begin{array}{c}
 i\mid k\\ 
 j\mid k\\
 k\leq n
\end{array}}
g(k)=\sum_{\footnotesize
\begin{array}{c}
[i,j]\mid k\\
k\leq n
\end{array}}g(k)=
\sum_{l\leq \frac{n}{[i,j]}}g\big([i,j]l\big).
\end{eqnarray*}
Because  $g(n)$ is totally multiplicative
\[
a_{ij}=g\big([i,j]\big)\sum_{\ell\leq \frac{n}{[i,j]}}g(\ell)=f\big([i,j]\big).
\]
If we calculate the determinant of both parts of (\ref{3_1_1}) we have (\ref{111}).
\end{proof}

\noindent
{\bf Particular cases}

\begin{example}
If $g(n)=1,$ then
\[
f\big(n,[i,j]\big)=\left\lfloor \frac{n}{[i,j]}\right\rfloor ,
\]
where $\lfloor x\rfloor$ denotes the integer part of $x$.
\\
From Theorem \ref{3_1_2} we have
\[
\left[\left\lfloor \frac{n}{[i,j]}\right\rfloor\right]_{n\times n}=
C_n^{T}\mbox{ diag}\big(1,\; 1,\ldots ,1\big)C_n,
\]
\[
\det \left[\left\lfloor \frac{n}{[i,j]}\right\rfloor\right]_{n\times n}=1.
\]
\end{example}

\begin{example}
If $g(n)=n,$ then
\[
f\big(n,[i,j]\big)=
\frac{\left\lfloor \frac{n}{[i,j]}\right\rfloor\left\lfloor \frac{n}{[i,j]}+1\right\rfloor}{2}.
\]
The decomposition of generalized LCM matrix is
\[
\big[f(n,[i,j])\big]_{n\times n}=
C_n^{T}\mbox{ diag}\big(1,\; 2,\ldots ,n\big)C_n,
\]
and the determinant
\[
\det \big[f(n,[i,j])\big]_{n\times n}=n!.
\]
\end{example}

\begin{example}
If $g(n)=(-1)^{\Omega (n)}$ is a Liouville function, then 
\[
f\big(n,[i,j]\big)=
(-1)^{\Omega\big([i,j]\big)}\sum_{k\leq \frac{n}{[i,j]}}(-1)^{\Omega(k)}
\]
and
\[
\big[f(n,[i,j])\big]_{n\times n}=
C_n^{T}\mbox{ diag}\big(1,\; -1,\ldots ,(-1)^{\Omega(n)}\big)C_n,
\]
\[
\det \big[f(n,[i,j])\big]_{n\times n}=(-1)^{\sum_{k=1}^n\Omega(k)}.
\]
\end{example}
We remark that matrices  related to the greatest integer function appeared in
\cite{jacobsthal1, carlitz1}.

\begin{thm}
\label{11}
For a given totally multiplicative function $g$ let
\[
f(n,i,j,[i,j])=\sum_{k\leq n}g(k)-g(i)\sum_{l\leq \frac{n}{i}}g(l)-g(j)\sum_{l\leq \frac{n}{j}}g(l)+g([i,j])\sum_{k\leq \frac{n}{[i,j]}}g(k).
\]
Then
\[
\big[f(n,i,j,[i,j])\big]_{n\times n}=D_n^T\mbox{ diag}[g(1),g(2),\ldots ,g(n)]D_n,
\]
where
$D_n=[d_{ij}]_{n\times n}$, 
\[
d_{ij}=
\left\{
\begin{array}{ccc}
1,&\mbox{ if }&j\not|\; i\\
0,&\mbox{ if }&j \mid  i
\end{array}
\right. .
\]
\end{thm}

\begin{proof}
After multiplication the general element of the matrix
\[
A=[a_{ij}]_{n\times n}=D_n^T\mbox{ diag}[g(1),g(2),\ldots ,g(n)]D_n
\]
is
\begin{eqnarray*}
a_{ij}&=&\sum_{
\begin{array}{c}
i\not| k\\
j\not|\; k\\
k\leq n
\end{array}}
g(k)=\sum_{k\leq n}g(k)-\sum_{i\mid k}g(k)-\sum_{j\mid k}g(k)+\sum_{
\begin{array}{c}
i\mid k\\
j\mid k\\
k\leq n
\end{array}}
g(k)=\\
&=&\sum_{k\leq n}g(k)-\sum_{il\leq n}g(il)-\sum_{jl\leq n}g(jl)+
\sum_{\footnotesize
\begin{array}{c}
[i,j]\mid k\\
k\leq n
\end{array}}g(k)
\end{eqnarray*}
The total multiplicativity of $g$ implies,
\begin{eqnarray*}
a_{ij}&=&
\sum_{k\leq n}g(k)-g(i)\sum_{l\leq \frac{n}{i}}g(l)-g(j)\sum_{l\leq \frac{n}{j}}g(l)+g([i,j])\sum_{k\leq \frac{n}{[i,j]}}g(k)=\\
&=&
f(n,i,j,[i,j]).
\end{eqnarray*}
\end{proof}

\noindent
{\bf Particular cases}

\begin{example}
If $g(n)=1,$ then
\[
f\big(n,i,j,[i,j]\big)=\tau(n)-\tau\left(\left\lfloor \frac{n}{i}\right\rfloor\right)-\tau\left(\left\lfloor\frac{n}{j}\right\rfloor\right)+
\left\lfloor \frac{n}{[i,j]}\right\rfloor ,
\]
where $\tau(n)=\displaystyle \sum_{d\mid n}1$.
By Theorem \ref{11}
\[
\left[f\left(n,i,j,\left\lfloor \frac{n}{[i,j]}\right\rfloor\right)\right]_{n\times n}=
D_n^{T}\mbox{ diag}\big(1,\; 1,\ldots ,1\big)D_n.
\]
\end{example}

\begin{example}
If $g(n)=n,$ then
\[
f\big(n,i,j,[i,j]\big)=\sigma(n)-\sigma\left(\left\lfloor \frac{n}{i}\right\rfloor\right)-\sigma\left(\left\lfloor\frac{n}{j}\right\rfloor\right)
+
\frac{\left\lfloor \frac{n}{[i,j]}\right\rfloor\left\lfloor \frac{n}{[i,j]}+1\right\rfloor}{2},
\]
where $\sigma(n)=\displaystyle \sum_{d\mid n}d$.
\\
The general form of a generalized LCM matrix is
\[
\big[f(n,i,j,[i,j])\big]_{n\times n}=
D_n^{T}\mbox{ diag}\big(1,\; 2,\ldots ,n\big)D_n.
\]
\end{example}

\begin{example}
If $g(n)=(-1)^{\Omega (n)}$ is the Liouville function then 
\begin{eqnarray*}
f\big(n,i,j,[i,j]\big)&=&\sum_{k\leq n}(-1)^{\Omega (k)}-(-1)^{\Omega (i)}\sum_{l\leq \frac{n}{i}}(-1)^{\Omega (l)}-
(-1)^{\Omega (j)}\sum_{l\leq \frac{n}{j}}(-1)^{\Omega (l)}g+\\
&+&
(-1)^{\Omega\big([i,j]\big)}\sum_{k\leq \frac{n}{[i,j]}}(-1)^{\Omega(k)}
\end{eqnarray*}
and
\[
\big[f(n,i,j,[i,j]\big]_{n\times n}=
D_n^{T}\mbox{ diag}\big(1,\; -1,\ldots ,(-1)^{\Omega(n)}\big)D_n.
\]
\end{example}

\begin{remark}
Due to the fact that the first line of the matrix $\![f(n,i,j,[i,j])]_{n\times n}$ contains only 0-s, the determinant of the matrix will always be 0.
\end{remark}

\subsection*{Acknowledgement}
This research was supported by the grant of Sapientia Foundation, Institute of Scientific Research.

\end{document}